\documentclass[reqno,a4paper,11pt]{amsart}

\usepackage[pdfpagelabels]{hyperref}
\usepackage{amssymb}
\usepackage{url}

\newtheorem{theorem}{Theorem}%[section]
\newtheorem{lemma}[theorem]{Lemma}
\newtheorem{corollary}[theorem]{Corollary}

\newtheorem{oldtheorem}{Theorem}

\theoremstyle{definition}

\theoremstyle{remark}

\newtheorem{example}{Example}

\newcommand{\D}{\mathbb{D}}

\newcommand{\N}{\mathbb{N}}
\newcommand{\Z}{\mathbb{Z}}
\newcommand{\R}{\mathbb{R}}

\newcommand{\C}{\mathbb{C}}

\renewcommand{\phi}{\varphi}

\begin{document}
\title[Oscillation of solutions of LDE's]{Oscillation of solutions of LDE's
in domains conformally equivalent to unit disc}
\thanks{The second author is supported in part by the Academy of Finland project \#286877.
The fourth author is supported in part by Ministerio de Econom\'ia y Competitivivad, Spain, 
projects MTM2014-52865-P and MTM2015-69323-REDT; and La Junta de Andaluc\'ia, project FQM210.}

\author{I.~Chyzhykov}
\address{Faculty of Mathematics and Computer Science\newline
\indent Warmia and Mazury University of Olsztyn\newline
\indent S\l{}oneczna 54, Olsztyn, 10710, Poland}               
\email{chyzhykov@matman.uwm.edu.pl}

\author{J.~Gr\"ohn}
\address{Department of Physics and Mathematics\newline
\indent University of Eastern Finland\newline 
\indent P.O. Box 111, FI-80101 Joensuu, Finland}
\email{\vspace*{-0.25cm}janne.grohn@uef.fi}

\author{J.~Heittokangas}
%\address{Department of Physics and Mathematics, University of Eastern Finland\newline
%\indent P.O. Box 111, FI-80101 Joensuu, Finland}
\email{\vspace*{-0.25cm} janne.heittokangas@uef.fi}

\author{J.~R\"atty\"a}
%\address{Department of Physics and Mathematics, University of Eastern Finland\newline
%\indent P.O. Box 111, FI-80101 Joensuu, Finland}
\email{jouni.rattya@uef.fi}

\date{\today}

% 34M10: Ordinary differential equations, Oscillation, growth of solutions
% 30D35: Functions of a complex variable, Distribution of values, Nevanlinna theory


\begin{abstract}
Oscillation of solutions of $f^{(k)} + a_{k-2} f^{(k-2)} + \dotsb + a_1 f' +a_0 f = 0$
is studied in domains conformally equivalent to the unit disc.
The  results are applied, for example, to Stolz angles, horodiscs, sectors and strips.
The method relies on a~new conformal transformation of higher order linear differential equations. 
Information on the existence of zero-free solution bases is also obtained.
\end{abstract}

\maketitle

\section{Introduction and results}

The classical univalence criterion due to Nehari \cite{N:1949} states that a~locally univalent 
meromorphic function $f$ in the unit disc $\D$ is one-to-one if its Schwarzian derivative
$S_f = (f''/f')'-(1/2)(f''/f')^2$ satisfies $|S_f(z)|(1-|z|^2)^2\leq 2$ for all $z\in\D$. Nehari's proof 
is based on the representation $a=S_{(f_1/f_2)}/2$ of the analytic coefficient of
\begin{equation} \label{eq:de2}
f''+af=0
\end{equation}
in terms of the quotient of its 
two linearly independent solutions $f_1$ and $f_2$. The proof further uses 
a~transformation of \eqref{eq:de2} into
\begin{equation} \label{eq:de2g}
  g''+bg=0,\quad
  b = (a\circ T)(T')^2 + S_T/2,
\end{equation}
where $T$ maps $\D$ conformally onto $\D$ and the functions 
$(f_1\circ T)(T')^{-1/2}$ and $(f_2\circ T)(T')^{-1/2}$ form a~solution base of~\eqref{eq:de2g}.
In fact, this method is independent of the underlying regions,
and can be performed between any two conformally equivalent domains.
Such transformations have turned out fundamental in many applications in the theory of differential equations,
and appear in \cite[p.~394]{I:1956} whose English edition was published in 1926.

Our first objective is to  transform the differential equation
\begin{equation} \label{eq:fdek}
  f^{(k)} + a_{k-2} f^{(k-2)} +a_{k-3} f^{(k-3)} + \dotsb + a_1 f' +a_0 f = 0, \quad k\geq 2,
\end{equation}
with analytic coefficients in a~domain $\Omega_1$, to another differential equation 
\begin{equation} \label{eq:gdek2}
  g^{(k)} + b_{k-2} g^{(k-2)} +b_{k-3} g^{(k-3)} + \dotsb + b_1 g' +b_0 g = 0,
\end{equation}
where the coefficients are analytic in a~domain $\Omega_2$, which is conformally equivalent to $\Omega_1$.
This transformation is given in terms of the incomplete exponential Bell polynomials %$B_{i,n}$ for $i\geq n$,
\begin{equation*}
  B_{i,n}\big( z_1,\dotsc,z_{i-n+1} \big) = \sum \frac{i!}{j_1! \, j_2! \dotsb j_{i-n+1}!} 
  \left( \frac{z_1}{1!} \right)^{j_1} \left( \frac{z_2}{2!} \right)^{j_2} \dotsb \left( \frac{z_{i-n+1}}{(i-n+1)!}\right)^{j_{i-n+1}},
\end{equation*}
where $i\geq n$ and the sum is taken over all sequences $j_1,j_2, \dotsc, j_{i-n+1}$ of
non-negative integers satisfying the equations
\begin{equation} \label{eq:faaeq}
  \begin{split}
    \left\{
      \begin{aligned}
        i & = j_1 + 2 j_2 + \dotsb + (i-n+1) j_{i-n+1},\\
        n & = j_1 + j_2 + \dotsb + j_{i-n+1}.
      \end{aligned}
    \right.
  \end{split}
\end{equation}
For example, by a~straight-forward computation
\begin{equation*}
  B_{i,i}(z_1) = (z_1)^i, 
  \quad 
  B_{i,i-1}(z_1,z_2) %= \frac{i!}{(i-2)! \cdot 1!} \, \left( \frac{z_1}{1!} \right)^{i-2} \, \left( \frac{z_2}{2!} \right)^1
  = \frac{i(i-1)}{2} \, (z_1)^{i-2} \, z_2
\end{equation*}
and
\begin{align*}
  B_{i,i-2}(z_1,z_2,z_3) 
  & = \frac{i(i-1)(i-2)}{3} \, z_1^{i-3} \, z_3 + \frac{i(i-1)(i-2)(i-3)}{4} \, z_1^{i-4} \, z_2^2.
\end{align*}

%%%%%%%%%%%%%%%%%%%%%%%%
%%%% ---- THEOREM ---- %%%%
%%%%%%%%%%%%%%%%%%%%%%%%

\begin{theorem} \label{thm:trans}
Let $T$ map $\Omega_2$ conformally onto $\Omega_1$,
and let $h=(T')^{(1-k)/2}$.
Suppose that $\{f_1,\dotsc,f_k\}$ is a~solution base of
the differential equation~\eqref{eq:fdek}, where the coefficients $a_0, \dotsc, a_{k-2}$
are analytic in $\Omega_1$. Then $\{(f_1 \circ T) h,\dotsc,(f_k \circ T) h \}$ is a~solution base 
of~\eqref{eq:gdek2}, where the coefficients $b_0,\dotsc, b_{k-2}$ are analytic in $\Omega_2$. Moreover,
\begin{equation} \label{al}
  \begin{split}
    (a_\ell \circ T)(T')^{k-\ell} & = \sum_{j=\ell}^{k-1} b_j \left[ \sum_{i=\ell}^j \binom{j}{i} \, 
      \frac{B_{i,\ell}\big( T', \dotsc, T^{(i-\ell+1)} \big)}{(T')^\ell} \, \frac{h^{(j-i)}}{h} \right]\\
    & \qquad + \sum_{i=\ell}^{k-1} \binom{k}{i} \frac{B_{i,\ell}\big( T', \dotsc, T^{(i-\ell+1)} \big)}{(T')^\ell} \, \frac{h^{(k-i)}}{h}
    + \frac{B_{k,\ell}\big( T', \dotsc, T^{(k-\ell+1)} \big)}{(T')^\ell}
  \end{split}
\end{equation}
for any $\ell\in \{ 1, \dotsc, k-2\}$, and
\begin{equation}\label{a0}
  (a_0 \circ T) \, (T')^k = \frac{h^{(k)}}{h} + b_{k-2} \, \frac{h^{(k-2)}}{h} + \dotsb + b_1 \frac{h'}{h} + b_0. 
\end{equation}
\end{theorem}

With appropriate modifications, the method of proof of
Theorem~\ref{thm:trans} applies, for example, in the case of real differential equations.

The representation~\eqref{al} for $\ell=k-2$ simplifies to
\begin{equation*} 
\begin{split}
(a_{k-2} \circ T) \, (T')^2 
& = b_{k-2} +  \frac{k(k-1)}{2} \left( \frac{h''}{h} \right) + \frac{k(k-1)(k-2)}{2} \left( \frac{T''}{T'} \right) \left( \frac{h'}{h} \right)\\
& \qquad + \frac{k(k-1)(k-2)}{3}  \left( \frac{T'''}{T'} \right) + \frac{k(k-1)(k-2)(k-3)}{4} \, \left( \frac{T''}{T'} \right)^2.
\end{split}
\end{equation*}
The particular case $k=2$ of this identity reduces to the situation in~\eqref{eq:de2g} and 
reveals the well-known connection between Bell polynomials and Schwarzian derivatives.

Let $T$ be a~conformal map from $\D$ into $\C$. The standard functions in Nevanlinna theory 
for a~function $f$ meromorphic in $T(\D)$ 
are defined to be the corresponding functions for~$f \circ T$. In particular,
\begin{equation*}
  N\Big( T\big(D(0,r) \big),0,f \Big) = N(r,0,f \circ T), \quad 0<r<1,
\end{equation*}
where $N(r,a,g)$ is the standard integrated counting function for the $a$-points of $g$ in the disc $D(0,r)=\{z\in\C : |z|<r\}$.

Our second objective is to
quantify the phenomenon that local growth of any coefficient of~\eqref{eq:fdek}
implies local oscillation for some non-trivial solution. In the proof we apply Theorem~\ref{thm:trans} in the case when $\Omega_2=\D$.

%%%%%%%%%%%%%%%%%%%%%%%%
%%%% ---- THEOREM ---- %%%%
%%%%%%%%%%%%%%%%%%%%%%%%

\begin{theorem}\label{Local-Disc-thm-general}
Let $T$ map $\D$ conformally into $\C$, $0<b<1$ and  $s(r)=1-b(1-r)$ for $0\leq r <1$.
Suppose that $\{f_1,\ldots,f_{k}\}$ is a~solution base of \eqref{eq:fdek}, where $a_0,\ldots,a_{k-2}$ are analytic in~$T(\D)$.
Then there exists a~constant $K=K(b)$ such that, for any $j\in \{0,\ldots,k-2\}$,
\begin{align*}
  \int_{T(D(0,r))}|a_j(z)|^\frac{1}{k-j}\, \frac{dm(z)}{| T'( T^{-1}(z) )|}
  & \leq K \Bigg(  \sum_{j=1}^{k} \int_0^{s(r)} \!\frac{N\big(T(D(0,t)),0,f_j\big)}{1-t} \, dt\\
  & \qquad +\sum_{j=1}^{k-1} \int_0^{s(r)} \!\frac{N\big(T(D(0,t)),0,f_j+f_k\big)}{1-t}\, dt +\log^2\frac{e}{1-r} \Bigg)
\end{align*}
outside a~possible exceptional set $E \subset [0,1)$ for which $\int_Edt/(1-t)<\infty$.
\end{theorem}

By \cite[Lemma~C]{B:1972}, for a~sufficiently small $0<b<1$
the statement of Theorem~\ref{Local-Disc-thm-general} 
is valid without any exceptional set. We may also suppose
    \begin{equation}\label{Eq:Assumption-Discs-general}
\limsup_{r\to1^-} \,  \frac{\displaystyle\int_{T(D(0,r))}|a_j(z)|^\frac{1}{k-j}\, \frac{dm(z)}{| T'( T^{-1}(z) )|}}{\log^2(e/(1-r))} = \infty,
    \end{equation}
for some $j\in\{ 0,\dotsc, k-2\}$, for otherwise the assertion is trivially valid.
The condition~\eqref{Eq:Assumption-Discs-general} 
guarantees the existence of a~solution of \eqref{eq:fdek} having more zeros in $T(\D)$ 
than any non-admissible analytic function in~$\D$.

%%%%%%%%%%%%%%%%%%%%%%%%
%%%% ---- THEOREM ---- %%%%
%%%%%%%%%%%%%%%%%%%%%%%%

\begin{corollary}\label{Local-Disc-cor-general}
Under the assumptions of Theorem~\ref{Local-Disc-cor-general}, there
exists $0<b<1$ and $K=K(b)$ such that
\begin{align*}
  &\frac{\displaystyle\int_{T(D(0,r))}|a_j(z)|^\frac{1}{k-j}\, \frac{dm(z)}{| T'( T^{-1}(z) )|}}{\log(e/(1-r))}\\
  &\qquad\leq K \Bigg( \sum_{j=1}^{k}N\Big(T\big(D(0,s(r))\big),0,f_j\Big)
  +\sum_{j=1}^{k-1}N\Big(T\big(D(0,s(r))\big),0,f_j+f_k\Big)+\log\frac{e}{1-r} \Bigg)
\end{align*}
for all $0 \leq r <1$.
\end{corollary}

Connections between the oscillation of solutions and
the growth of analytic coefficients have been thoroughly studied 
in the cases of $\D$ and $\C$.
However, the existing literature contains only scattered results on local oscillation of solutions in standard 
regions such as Stolz angles, horodiscs, sectors and strips. We next show that, for appropriate choices 
of~$T$, Theorem~\ref{Local-Disc-thm-general} yields new information in these particular regions.

{\em Stolz angles.}
Fix $0<\alpha<1$ and $\zeta\in\partial\D$, and
let $T(z)=\zeta( 1 - (1-z\overline{\zeta}\, )^\alpha )$ for all $z\in\D$. Then $T(\D)\subset\D$ and 
$\partial T(\D)$ takes the form of a~petal which
has a~corner of opening $\alpha\pi$ at $T(\zeta)=\zeta$. In particular, the domain $T(\D)$
can be seen as a Stolz angle with vertex at $\zeta$. In this
case $|T'(T^{-1}(z))| = \alpha \, | \zeta - z|^{1-1/\alpha}$ for all $z\in T(\D)$.

{\em Horodiscs.}
Fix $\zeta\in\partial\D$, and let $T(z) = \zeta + (1-|\zeta|) z$ for all $z\in\D$.
Then $T(\D) \subset \D$ and $\partial T(\D)$ is a~circle internally tangent to $\partial\D$
at $\zeta$. Now $|T'(T^{-1}(z))| = 1-|\zeta|$ for all $z\in T(\D)$.

{\em Sectors.}
Fix $\varphi\in\R$ and $0<\alpha<2$, and let $T(z) = e^{i\varphi} ( (1+z)/(1-z))^\alpha$ for all $z\in\D$.
Then $T(\D)$ is a~sector of opening $\alpha\pi/2$, in the direction $\varphi$, and
\begin{equation*}
  \big|T'(T^{-1}(z))\big| = \frac{\alpha}{2} \, |z|^{1-1/\alpha} \, \big| z^{1/\alpha} + e^{i\varphi/\alpha} \big|^2, \quad z\in T(\D).
\end{equation*}

{\em Strips.}
Fix $\varphi\in\R$ and $0<\alpha<\infty$, and let 
$T(z) = \alpha e^{i\varphi} \log((1+z)/(1-z))$ for all $z\in\D$.
Then $T(\D)$ is a~strip of width $\alpha\pi$, and
\begin{equation*}
  \big|T'(T^{-1}(z))\big| = \frac{\alpha}{2} \, \big| e^{z/\alpha} + e^{i\varphi} \big|^2 \, e^{-{\rm Re} (z/\alpha)}, \quad z\in T(\D).
\end{equation*}

The next result combined with
\cite[p.~356]{BL:1982}
shows that
the solutions $f_1,\dotsc, f_k$
in Theorem~\ref{Local-Disc-thm-general} can be zero-free, while the coefficients may grow arbitrarily fast.
This implies, in particular, that the second sum in the upper bound cannot be removed.

%%%%%%%%%%%%%%%%%%%%%%%%%%%
%%%% ---- THEOREM ---- %%%%
%%%%%%%%%%%%%%%%%%%%%%%%%%%

\begin{theorem} \label{prop:new}
Suppose that $f_1$ and $f_2$ are linearly independent solutions of $f''+af=0$, where the coefficient $a$ is analytic. For any
$k\geq 2$, the functions $f_1^{k-1}, \, f_1^{k-2} f_2^{\phantom{k}}, \dotsc, f_1^{\phantom{k}} f_2^{k-2}, \, f_2^{k-1}$
are linearly independent solutions of \eqref{eq:fdek} with analytic coefficients
$a_0,\dotsc, a_{k-2}$. Moreover,
\begin{equation} \label{eq:akm2}
  a_{k-2} = \binom{k+1}{k-2} \, a
  = \frac{(k-1)k(k+1)}{6} \, a.
\end{equation}
\end{theorem}

In general,
if all solutions of
\begin{equation*}
  f^{(k)} + a_{k-1} f^{(k-1)} + a_{k-2} f^{(k-2)} + \dotsb + a_1 f' +a_0 f = 0
\end{equation*}
are meromorphic, then the coefficients $a_0, \dotsc, a_{k-1}$ are uniquely determined 
meromorphic functions which can be represented in terms of
Wronskian type determinants of any $k$ linearly independent solutions \cite[Proposition~1.4.6]{L:1993}.
In particular, if $f_1$ and $f_2$ are linearly independent solutions of $f''+af=0$, 
then \cite[Proposition~D]{I:1994} implies that $f_1^2,f_1^{\phantom{1}} f_2^{\phantom{1}}, f_2^2$ are linearly 
independent solutions of $f''' + 4a f' + 2 a' f =0$. By a~straight-forward computation,
it can be verified that
$f_1^3,f_1^2f_2^{\phantom{1}},f_1^{\phantom{1}}f_2^2,f_2^3$ are linearly independent solutions of
\begin{equation*}
  f^{(4)} + 10 a f'' + 10 a' f' + (3a'' + 9 a^2)f = 0,
\end{equation*}
which reveals the exact coefficients in the case $k=4$.

The remaining part of this paper is organized as follows. Theorem~\ref{thm:trans} is
proved in  Section~\ref{sec:trans}. Section~\ref{sec:aux} contains
auxiliary results, which are needed in the proof of Theorem~\ref{Local-Disc-thm-general} in
Section~\ref{sec:proof}. Sharpness of Theorem~\ref{Local-Disc-thm-general} is illustrated in 
Section~\ref{sec:examples}. Theorem~\ref{prop:new} is proved in Section~\ref{sec:bases}.

%%%%%%%%%%%%%%%%%%%%%%%
%%%% ---- SECTION ---- %%%%
%%%%%%%%%%%%%%%%%%%%%%%

\section{Proof of Theorem~\ref{thm:trans}} \label{sec:trans}

In the following argument some details related to straight-forward calculations are omitted.
Let $f$ be a~solution of~\eqref{eq:fdek} and $g = (f \circ T) h$, where $h= (T')^{(1-k)/2}$. Since
\begin{equation*}
  g^{(j)} = \sum_{i=0}^j \binom{j}{i} (f \circ T)^{(i)} h^{(j-i)}, \quad j\in\N,
\end{equation*}
by the general Leibniz rule, Fa\`a di Bruno's formula gives
\begin{equation} \label{eq:faa}
  g^{(j)} 
  = (f \circ T) h^{(j)} 
  + \sum_{i=1}^j \binom{j}{i} \left( \, \sum_{n=1}^i (f^{(n)} \circ T)  \, B_{i,n}\big(T', \dotsc, T^{(i-n+1)} \big) \right) h^{(j-i)}, 
\quad j\in\N.
\end{equation}
We proceed to determine the coefficients $b_0, \dotsc, b_{k-1}$ such that
\begin{equation} \label{eq:gdek}
g^{(k)} + b_{k-1} g^{(k-1)} + b_{k-2} g^{(k-2)} + \dotsb + b_1 g' +b_0 g = 0.
\end{equation}
On one hand, the differential equation \eqref{eq:gdek} implies
\begin{align*}
  -g^{(k)} 
  & =  \sum_{j=1}^{k-1} b_{j} \left[ (f \circ T) h^{(j)} 
    + \sum_{i=1}^{j} \binom{j}{i} \left( \, \sum_{n=1}^i (f^{(n)} \circ T)  \, B_{i,n}\big(T', \dotsc, T^{(i-n+1)} \big) \right) h^{(j-i)} \right]\\
  & \qquad + b_0 \big[ (f \circ T) h \big].
\end{align*}
On the other hand, by applying \eqref{eq:faa} for $g^{(k)}$ and then taking advantage of \eqref{eq:fdek}, we deduce
\begin{align*}
  -g^{(k)} & = - (f \circ T) h^{(k)} - \sum_{i=1}^{k-1} \binom{k}{i} \left( \, \sum_{n=1}^i (f^{(n)} \circ T)  
    \, B_{i,n}\big(T', \dotsc, T^{(i-n+1)} \big) \right) h^{(k-i)}\\
  &  \qquad - \left( \, \sum_{n=1}^{k-1} (f^{(n)} \circ T)  
    \, B_{k,n}\big(T', \dotsc, T^{(k-n+1)} \big) \right) h 
  - (f^{(k)} \circ T)  
  \, B_{k,k}\big(T' \big) h\\
  & = - (f \circ T) h^{(k)} - \sum_{i=1}^{k-1} \binom{k}{i} \left( \, \sum_{n=1}^i (f^{(n)} \circ T)  
    \, B_{i,n}\big(T', \dotsc, T^{(i-n+1)} \big) \right) h^{(k-i)}\\
  &  \qquad - \left( \, \sum_{n=1}^{k-1} (f^{(n)} \circ T)  
    \, B_{k,n}\big(T', \dotsc, T^{(k-n+1)} \big) \right) h 
  + 
  \sum_{j=0}^{k-2} (a_{j} \circ T)( f^{(j)} \circ T )  B_{k,k}\big(T' \big) h.
\end{align*}
By comparing the coefficients of $f^{(k-1)} \circ T$, we get
\begin{equation*}
   b_{k-1} \, \binom{k-1}{k-1} B_{k-1,k-1}\big( T' \big) \, h = - \binom{k}{k-1} \, B_{k-1,k-1}\big( T' \big) h' - B_{k,k-1}\big( T',T'' \big) \, h,
\end{equation*}
where the right-hand side reduces to
\begin{align*}
  & - k \, (T')^{k-1} h' - \frac{(k-1)(k-2)}{2} \, (T')^{k-3} \, T'' \, h \\
  & \qquad = - k \, (T')^{k-1} \, \frac{1-k}{2} \, (T')^{\frac{1-k}{2} - 1} T''- \frac{k(k-1)}{2} \, (T')^{k-2} \, T'' \, (T')^{\frac{1-k}{2}} \equiv 0.
\end{align*}
Therefore $b_{k-1}\equiv 0$ and \eqref{eq:gdek} reduces to \eqref{eq:gdek2}.
By comparing the coefficients of $f^{(\ell)} \circ T$ for $\ell\in \{ 1, \dotsc, k-2\}$, we get
\begin{align*}
  & \sum_{j=\ell}^{k-1} b_j \left[ \sum_{i=\ell}^j \binom{j}{i} B_{i,\ell}\big( T', \dotsc, T^{(i-\ell+1)} \big) h^{(j-i)} \right]\\
  & \qquad = - \sum_{i=\ell}^{k-1} \binom{k}{i} B_{i,\ell}\big( T', \dotsc, T^{(i-\ell+1)} \big) h^{(k-i)}
  - B_{k,\ell}\big( T', \dotsc, T^{(k-\ell+1)} \big) h\\
  & \qquad  \qquad 
  + (a_\ell \circ T)B_{k,k}(T') h.
\end{align*}
Since $B_{k,k}(T') = (T')^{k-\ell} (T')^\ell$, we deduce~\eqref{al} for any $\ell\in \{ 1, \dotsc, k-2\}$. 
By comparing the coefficients of $f \circ T$, we get
\begin{equation*}
  b_{k-2} \, h^{(k-2)} + \dotsb + b_1 h' + b_0 h = -h^{(k)} + (a_0 \circ T) \, B_{k,k}(T')\, h,
\end{equation*}
which implies~\eqref{a0}. Since the statement concerning solution bases is trivial, Theorem~\ref{thm:trans} is now proved.

%%%%%%%%%%%%%%%%%%%%%%%
%%%% ---- SECTION ---- %%%%
%%%%%%%%%%%%%%%%%%%%%%%

\section{Auxiliary results} \label{sec:aux}

The proof of Theorem~\ref{Local-Disc-thm-general} depends on
several auxiliary results, which are considered next.

%%%%%%%%%%%%%%%%%%%%%%
%%%% ---- LEMMA ---- %%%%
%%%%%%%%%%%%%%%%%%%%%%

\begin{lemma}\label{lemma:logderivative}
Let $j$ and $k$ be integers with $k>j\geq 0$, and let $f$ be a
meromorphic function in $\D$ such that $f^{(j)}\not\equiv 0$. 
Let $0<b<1$, and write $s(r)=1-b(1-r)$ for $0\leq r <1$.
Then there exists a constant $K=K(b)>0$ such that 
\begin{equation*}
  \int_{D(0,r)}\bigg|\frac{f^{(k)}(z)}{f^{(j)}(z)}\bigg|^\frac{1}{k-j}\, dm(z)
  \leq K\, \bigg(\max_{j\leq m\leq k-1}\int_0^{s(r)} \frac{T(t,f^{(m)})}{1-t}\, dt+\log\frac{e}{1-r}\bigg),
  \quad 0\leq r<1.
\end{equation*}
\end{lemma}

%%%%%%%%%%%%%%%%%%%%%%
%%%% ---- PROOF ---- %%%%
%%%%%%%%%%%%%%%%%%%%%%

\begin{proof}
For $0<r_1<r_2<1$, let $A(r_1,r_2)=\{z \in\D : r_1< |z|\leq r_2\}$.
Let $0<d<1$ be a~constant which will be fixed later, and define $R_\nu= R_\nu(d) = 1-d^{\nu}$ for $\nu\in\N$.
The proof of \cite[Theorem~2.3(b)]{HR:2011} gives
\begin{equation*} 
  \int_{|z|\leq R_1}\bigg|\frac{f^{(k)}(z)}{f^{(j)}(z)}\bigg|^\frac{1}{k-j}\, dm(z)\leq C_1,
\end{equation*}
where $C_1=C_1(d,j,k)$ is a~constant.
Let $R_1 < r < 1$ and take $\mu=\mu(d) \in\N$ such that $R_\mu<r\leq R_{\mu+1}$, which is equivalent to
\begin{equation}\label{log-eq}
  \mu\log\frac{1}{d}<\log\frac{1}{1-r}\leq (\mu+1)\log\frac{1}{d}.
\end{equation}
The reasoning used in the proof of \cite[Theorem~5]{CHR:2009} yields
\begin{equation}\label{from-CHR-eq}
  \int_{A(R_\nu,R_{\nu+1})} \left| \frac{f'(z)}{f(z)} \right|\, dm(z) \le
  C_2\big( \, T(R_{\nu+3},f)+1\big),\quad \nu\in\N,
\end{equation}
where $C_2=C_2(d)$ is a constant independent of $\nu$. By \eqref{log-eq} and \eqref{from-CHR-eq}, we deduce
\begin{equation}\label{sum-annuli}
  \begin{split}
    \int_{A(R_1,r)} \left| \frac{f'(z)}{f(z)} \right|\, dm(z)
    &\leq C_2\, \sum_{j=1}^{\mu}\big( \, T(R_{j+3},f)+1\big)\\
    &= C_2 \, \Bigg(\frac{1}{1-d} \, \sum_{j=1}^{\mu}\frac{T(R_{j+3},f)}{1-R_{j+3}}\big(R_{j+4}-R_{j+3}\big)+\mu\Bigg)\\
    &\leq C_3\left( \, \int_{R_1}^{R_{\mu+4}}\frac{T(t,f)}{1-t}\, dt+\log\frac{1}{1-r}\right),
  \end{split}
\end{equation}
where $C_3=C_3(d)$ is a~constant such that $C_3=C_2\cdot\max \{ 1/(1-d),-1/\log d \}$.

Next we use the H\"older inequality and \eqref{sum-annuli} to conclude that
\begin{equation*} %\label{annulus-eq}
  \begin{split}
    \int_{A(R_1,r)}
    \bigg|\frac{f^{(k)}(z)}{f^{(j)}(z)}\bigg|^{\frac{1}{k-j}} \, dm(z)
    &= \int_{A(R_1,r)}
    \prod_{m=j}^{k-1}\bigg|\frac{f^{(m+1)}(z)}{f^{(m)}(z)}\bigg|^{\frac{1}{k-j}} dm(z)\\
    &\le\prod_{m=j}^{k-1}\left(\int_{A(R_1,r)}\bigg|\frac{f^{(m+1)}(z)}{f^{(m)}(z)}\bigg| \, dm(z)\right)^{{\frac{1}{k-j}}}\\
    &\le C_4\prod_{m=j}^{k-1}\left(\int_{R_1}^{R_{\mu+4}}\frac{T(t,f^{(m)})}{1-t}\, dt+\log\frac{1}{1-r}\right)^\frac{1}{k-j}\\
    &\le C_4\left(\max_{j\leq m\leq k-1}\int_{R_1}^{R_{\mu+4}}\frac{T(t,f^{(m)})}{1-t}\, dt+\log\frac{1}{1-r}\right),
  \end{split}
\end{equation*}
where $C_4=C_4(d,j,k)$ is a~constant. Note that
    \begin{equation*}
    R_{\mu+4}=1-d^4d^\mu=1-d^4(1-R_\mu)<1-d^4(1-r).
    \end{equation*}
Choose $0<d<1$ such that $b=d^4$. The assertion follows.
\end{proof}

For $1\leq \alpha<\infty$, let 
\begin{equation*}
  f(z)=\exp \left( - \left( \frac{1+z}{1-z} \right)^\alpha \, \right), \quad z\in\D.
\end{equation*}
If $\alpha=1$,
then~$f$ is an~atomic singular inner function and the Nevanlinna characteristic of~$f$ and all its derivatives are bounded. Therefore
all terms in the statement of Lemma~\ref{lemma:logderivative} are asymptotically comparable to $-\log(1-r)$
as $r\to 1^-$. Meanwhile, if $\alpha>1$, then both sides are of growth $(1-r)^{1-\alpha}$ as $r\to 1^-$. This illustrates the sharpness of
Lemma~\ref{lemma:logderivative}.

The following result allows
us to represent the coefficients in terms of quotients of linearly independent solutions.

%%%%%%%%%%%%%%%%%%%%%%%%%%%%
%%%% ---- OLD THEOREM ---- %%%%
%%%%%%%%%%%%%%%%%%%%%%%%%%%%

\begin{oldtheorem}[\protect{\cite[Theorem~2.1]{K:1969}}]\label{Kim-thm}
Let $g_1,\ldots,g_k$ be linearly
independent solutions of \eqref{eq:gdek2}, where
$b_0,\ldots,b_{k-2}$ are analytic in $\D$. Let
\begin{equation}\label{ratios-plane}
  y_1=\frac{g_1}{g_k},\ \ldots ,\ y_{k-1}=\frac{g_{k-1}}{g_k},
\end{equation}
and 
\begin{equation}\label{determinant-plane}
  W_j=\left|
    \begin{array}{cccc}
      y_1'\ & y_2'\ & \cdots \ & y_{k-1}'\\
      \vdots & \vdots & \ddots& \vdots \\
      y_1^{(j-1)}\ & y_2^{(j-1)}\ & \cdots & y_{k-1}^{(j-1)}\\
      y_1^{(j+1)}\ & y_2^{(j+1)}\ & \cdots & y_{k-1}^{(j+1)}\\
      \vdots & \vdots& \ddots &\vdots\\
      y_1^{(k)}\ & y_2^{(k)}\ & \cdots & y_{k-1}^{(k)}
    \end{array}
  \right|,\quad j=1,\ldots, k.
\end{equation}
Then
\begin{equation}\label{coefficients-rep-plane}
  b_j=\sum_{i=0}^{k-j} (-1)^{2k-i}\delta_{ki}\left(\begin{array}{c} k-i\\ k-i-j\end{array}\right)
  \frac{W_{k-i}}{W_k} \frac{\left(\sqrt[k]{W_k}\right)^{(k-i-j)}}{\sqrt[k]{W_k}},\quad j=0,\ldots,k-2,
\end{equation}
where $\delta_{kk}=0$ and $\delta_{ki}=1$ otherwise.
\end{oldtheorem}

We also need an estimate in the spirit of Frank-Hennekemper and Petrenko.

%%%%%%%%%%%%%%%%%%%%%%
%%%% ---- LEMMA ---- %%%%
%%%%%%%%%%%%%%%%%%%%%%

\begin{lemma}\label{mero-coeffs}
Let $g_1,\ldots, g_k$ be linearly independent meromorphic solutions of~\eqref{eq:gdek}
with coefficients $b_0,\ldots,b_{k-1}$ meromorphic in $\D$, and let $0<b<1$.
Then there exists a constant $K=K(b)>0$ such that
	\begin{equation*}
  	\int_{D(0,r)} |b_j(z)|^\frac{1}{k-j}\, dm(z)
  	\leq K 
  	\left(\max_{1\leq l\leq k}\, \int_0^{s(r)} \frac{T\big(t,g_l\big)}{1-t} \, dt +\log^2\frac{e}{1-r}\right),
  	\quad 0\leq r < 1,
	\end{equation*}
for all $j=0,\ldots,k-1$.
\end{lemma}

The statement in Lemma~\ref{mero-coeffs} for the equation $g^{(k)}+b_0g=0$ follows immediately from Lemma~\ref{lemma:logderivative} 
and the fact that
	\begin{equation}\label{derivative}
	\begin{split}
	T(r,g^{(j)})&\leq (j+1)N(r,g)+m\big(r,g^{(j)}\big)
         \leq (j+1)T(r,g)+m\big(r,g^{(j)}/g\big)\\
	&\lesssim T\big(s(r),g\big)+\log\frac{e}{1-r},\quad j\in\N.
	\end{split}	
	\end{equation}
The general case is a modification of \cite[Lemma~11]{CGHR} or of \cite[Lemma~7.7]{L:1993}.

%%%%%%%%%%%%%%%%%%%%%%%
%%%% ---- SECTION ---- %%%%
%%%%%%%%%%%%%%%%%%%%%%%

\section{Proof of Theorem~\ref{Local-Disc-thm-general}} \label{sec:proof}

Let $h=(T')^{(1-k)/2}$. If $f$ is a~solution of \eqref{eq:fdek}, then $g=(f\circ T)h$ is a solution of 
\eqref{eq:gdek2}. Based on this transformation, let $\{g_1,\ldots,g_k\}$ be 
a~solution base of \eqref{eq:gdek2} corresponding to the solution base $\{f_1,\ldots,f_k\}$ of \eqref{eq:fdek}. By the conformal change of variable,
	\begin{equation}\label{conformal-change}
	\begin{split}
	\int_{T(D(0,r))}|a_j(z)|^\frac{1}{k-j}\, \frac{dm(z)}{| T'( T^{-1}(z) )|}
	&=\int_{D(0,r)} \big|a_j(T(z))\big|^\frac{1}{k-j}\, \frac{| T'(z)|^2}{| T'(z)|}\, dm(z)\\
	&=\int_{D(0,r)} \big|a_j(T(z))\, T'(z)^{k-j}\big|^\frac{1}{k-j}\, dm(z)
	\end{split}	
	\end{equation}
for $j=0,\ldots,k-2$. 

\textsc{Case $j=0$.}
From \eqref{a0}, we have
	\begin{equation}\label{0}
	\big|(a_0 \circ T) \, (T')^k\big|^\frac{1}{k} 
	\leq \bigg|\frac{h^{(k)}}{h}\bigg|^\frac{1}{k} 
	+ \bigg|b_{k-2} \, \frac{h^{(k-2)}}{h}\bigg|^\frac{1}{k} 
	+ \dotsb + \left|b_1 \frac{h'}{h}\right|^\frac{1}{k} + |b_0|^\frac{1}{k}.
	\end{equation}
Since $T$ is univalent, it belongs to the Hardy space $H^p$ for $0<p<1/2$ by \cite[Theorem~3.16]{Duren}, and hence
$T$ is of bounded Nevanlinna characteristic. Therefore all derivatives are non-admissible in the sense that
	$$
	T\big(r,T^{(j)}\big)=O\!\left(\log\frac{e}{1-r}\right), \quad j\in\N.
	$$
Thus $h$ and all of its derivatives are non-admissible as well. Using Lemma~\ref{lemma:logderivative}, we obtain
	$$
	\int_{D(0,r)}\bigg|\frac{h^{(j)}(z)}{h(z)}\bigg|^\frac{1}{j}\, dm(z)=O\!\left(\log^2\frac{e}{1-r}\right), \quad j\in\N.
	$$
Hence, making use of \eqref{0} and H\"older's inequality with conjugate indices $p=k/(k-j)$ and $q=k/j$, we infer
	\begin{equation}\label{01}
	\begin{split}
	& \int_{D(0,r)} \big|a_0(T(z))\, T'(z)^{k}\big|^\frac{1}{k}\, dm(z)\\
	& \qquad \leq \sum_{j=1}^{k-2}\left(\int_{D(0,r)}|b_j(z)|^\frac{1}{k-j}\, dm(z) \right)^\frac{k-j}{k}
	\left(\int_{D(0,r)}\bigg|\frac{h^{(j)}(z)}{h(z)}\bigg|^\frac{1}{j}\, dm(z)\right)^\frac{j}{k}\\
	& \qquad \qquad 
        +\int_{D(0,r)}|b_0(z)|^\frac{1}{k}\, dm(z)+O\!\left(\log^2\frac{e}{1-r}\right)\\
	& \qquad \leq \sum_{j=1}^{k-2}\left(\int_{D(0,r)}|b_j(z)|^\frac{1}{k-j}\, dm(z) \right)^\frac{k-j}{k}
	O\!\left(\log^\frac{2j}{k}\frac{e}{1-r}\right)\\
	&\qquad \qquad +\int_{D(0,r)}|b_0(z)|^\frac{1}{k}\, dm(z)+O\!\left(\log^2\frac{e}{1-r}\right),
	\end{split}	
	\end{equation}
where the sums are empty if $k=2$. Let $y_1,\ldots,y_{k-1}$ be defined by
\eqref{ratios-plane}. By restating \cite[Proposition 1.4.7]{L:1993} with the aid of some basic
properties satisfied by Wronskian determinants \cite[Chapter~1.4]{L:1993}, we see that the functions
$1,y_1,\ldots, y_{k-1}$ are linearly independent meromorphic
solutions of the differential equation
    \begin{equation*}
    y^{(k)}-\frac{W_{k-1}(z)}{W_k(z)}y^{(k-1)}+\cdots + (-1)^{k+1}\frac{W_1(z)}{W_k(z)}y'=0,
    \end{equation*}
where $W_j$ are defined by \eqref{determinant-plane}. 
From Lemma~\ref{mero-coeffs} we now conclude
	\begin{equation}\label{wronskian-ratios-plane}
  	\int_{D(0,r)}\left|\frac{W_{k-i}(z)}{W_k(z)}\right|^\frac{1}{i} dm(z)
  	\lesssim \max_{1\leq l\leq k-1} \int_0^{s(r)} \frac{T(t,y_l)}{1-t}\, dt +\log^2\frac{e}{1-r}
	\end{equation}
for $i=1,\ldots, k-1$.
Moreover, Lemma~\ref{lemma:logderivative} yields
	\begin{equation}\label{Wk-logderivatives}
  	\begin{split}
    \int_{D(0,r)}\left|\frac{\left(\sqrt[k]{W_k}\right)^{(k-i-j)}(z)}
    {\sqrt[k]{W_k}(z)}\right|^\frac{1}{k-i-j} \, dm(z)
    &\lesssim \int_0^{s(r)} \frac{T(t,W_k)}{1-t} \, dt +\log^2\frac{e}{1-r}\\
    & \lesssim \max_{1\leq l\leq k-1} \int_0^{s(r)} \frac{T(t,y_l)}{1-t} \, dt +\log^2\frac{e}{1-r},
    %&
    %\lesssim \log\frac{e}{1-r}\left(\max_{1\leq %l\leq k-1}T\big(s(r),y_l\big)+\log\frac{e}{1-r}%\right),
  \end{split}
\end{equation}
where $i$ and $j$ are as in \eqref{coefficients-rep-plane}, and where \eqref{derivative} has
been used with $y_l$ in place of $g$. Writing the coefficients $b_j$
in the form \eqref{coefficients-rep-plane}, we deduce
	\begin{equation*}
  	|b_j|^\frac{1}{k-j}\lesssim 
  	\left|\frac{\left(\sqrt[k]{W_k}\right)^{(k-j)}}{\sqrt[k]{W_k}}\right|^\frac{1}{k-j}
  	+\sum_{i=1}^{k-j}\left|\frac{W_{k-i}}{W_k}\right|^\frac{1}{k-j}
  	\left|\frac{\left(\sqrt[k]{W_k}\right)^{(k-i-j)}}{\sqrt[k]{W_k}}\right|^\frac{1}{k-j}.
	\end{equation*}
Finally, we make use of \eqref{wronskian-ratios-plane} and \eqref{Wk-logderivatives} together with H\"older's inequality with conjugate indices $p=(k-j)/i$ and $q=(k-j)/(k-i-j)$,
$1\leq i< k-j$, ($i=k-j$ is a removable triviality), and conclude
	$$
	\int_{D(0,r)}|b_j(z)|^\frac{1}{k-j}\, dm(z)
	\lesssim \max_{1\leq l\leq k-1} \int_0^{s(r)} \frac{T(t,y_l)}{1-t} \, dt +\log^2\frac{e}{1-r},
     \quad j=0,\ldots,k-2.
	$$
Substituting this into \eqref{01} we obtain
\begin{equation}\label{a02}
  \int_{D(0,r)}|a_0(T(z))T'(z)^{k}|^\frac{1}{k}\, dm(z)
  \lesssim \max_{1\leq l\leq k-1} \int_0^{s(r)} \frac{T(t,y_l)}{1-t} \, dt +\log^2\frac{e}{1-r}.
\end{equation}
According to the second main theorem of Nevanlinna,
	$$
	T(r,y_l)\leq N(r,0,y_l)+N(r,\infty,y_l)+N(r,-1,y_l)+S(r,y_l),\quad r\not\in E,
	$$
where $S(r,y_l)=O\!\left(\log^+ T(r,y_l)-\log(1-r)\right)$, $l\in\{1,\ldots,k-1\}$ and 
the exceptional set~$E$ satisfies $\int_Edt/(1-t)<\infty$. Thus
	\begin{align*}
	T(r,y_l) &\leq 
	2N(r,0,g_l)+2N(r,0,g_k)+2N(r,0,g_l+g_k)+O\!\left(\log\frac{e}{1-r}\right)\\	
	&\leq  2N\Big(T\big(D(0,r)\big),0,f_l\Big)+2N\Big(T\big(D(0,r)\big),0,f_k\Big)\\
	& \qquad +2N\Big(T\big(D(0,r)\big),0,f_l+f_k\Big)+O\!\left(\log\frac{e}{1-r}\right),\quad r\not\in E,
	\end{align*}
for $l\in\{1,\ldots,k-1\}$. Combining this with \eqref{a02}, the assertion in the case $j=0$ follows.

\textsc{Case $j=\ell$, $1\leq \ell \leq k-2$.}
From \eqref{al}, we have
\begin{align}
  \big|(a_{\ell} \circ T) \, (T')^{k-\ell}\big|^\frac{1}{k-\ell}
  & \leq  \sum_{j=\ell}^{k-1} \left| b_j \right|^\frac{1}{k-\ell} \left( \sum_{i=\ell}^j \binom{j}{i}  
    \left| \frac{B_{i,\ell}\big( T', \dotsc, T^{(i-\ell+1)} \big)}{(T')^\ell} \right|^\frac{1}{k-\ell} \left| \frac{h^{(j-i)}}{h} \right|^\frac{1}{k-\ell} \right)\notag\\
  & \qquad + \sum_{i=\ell}^{k-1} \binom{k}{i} \left| \frac{B_{i,\ell}\big( T', \dotsc, T^{(i-\ell+1)} \big)}{(T')^\ell} \right|^\frac{1}{k-\ell} \left| \frac{h^{(k-i)}}{h} \right|^\frac{1}{k-\ell}\notag\\
  & \qquad + \left| \frac{B_{k,\ell}\big( T', \dotsc, T^{(k-\ell+1)} \big)}{(T')^\ell} \right|^\frac{1}{k-\ell}. \label{eq:finalterm}
\end{align}
We apply H\"older's inequality to estimate
\begin{equation*}
  \int_{D(0,r)} \big|a_{\ell}\big(T(z)\big) \, T'(z)^{k-\ell}\big|^\frac{1}{k-\ell}\, dm(z),
\end{equation*}
and content ourselves with writing details on the integration of the final term~\eqref{eq:finalterm} only. 
Since the Bell indices $j_1,j_2, \dotsc, j_{k-\ell+1}$ satisfy~\eqref{eq:faaeq} for $i=k$ and $n=\ell$, we obtain
\begin{align*}
  & \int_{D(0,r)} \left| \frac{B_{k,\ell}\big( T', \dotsc, T^{(k-\ell+1)} \big)}{(T')^\ell} \right|^\frac{1}{k-\ell}  \, dm(z)\\
  & \qquad \leq 
  \sum \frac{k!}{j_1! \dotsc j_{k-\ell+1}!}\int_{D(0,r)}  %\left| \frac{T'(z)}{1! \, T'(z)} \right|^\frac{j_1}{k-\ell} 
  \left| \frac{T''(z)}{2! \, T'(z)} \right|^\frac{j_2}{k-\ell} \dotsc 
  \left| \frac{T^{(k-\ell+1)}(z)}{(k-\ell+1)! \, T'(z)} \right|^\frac{j_{k-\ell+1}}{k-\ell}\, dm(z).
\end{align*}
Note that $k-\ell = j_2 + 2 j_3 + \dotsb + (k-\ell) j_{k-\ell+1}$.
The following application of H\"older's inequality is presented in the case that all Bell indices $j_1,j_2, \dotsc, j_{k-\ell+1}$
are non-zero. If there are zero indices, then the argument should be modified appropriately.
Choose the H\"older exponents
\begin{equation*}
  p_1 = \frac{k-\ell}{j_2} \geq 1, \quad p_2 = \frac{k-\ell}{2 \, j_3} \geq 1, \quad 
  \dotsc \quad p_{k-\ell} = \frac{k-\ell}{(k-\ell)\, j_{k-\ell+1}} = \frac{1}{j_{k-\ell+1}} \geq 1,
\end{equation*}
which satisfy
\begin{equation*}
  \frac{1}{p_1} + \dotsb + \frac{1}{p_{k-\ell}} = \frac{j_2 + 2 j_3 + \dotsb + (k-\ell) j_{k-\ell+1}}{k-\ell} = 1.
\end{equation*}
By H\"older's inequality,
\begin{align*}
  & \int_{D(0,r)} \left| \frac{T''(z)}{T'(z)} \right|^\frac{j_2}{k-\ell} \dotsc 
  \left| \frac{T^{(k-\ell+1)}(z)}{T'(z)} \right|^\frac{j_{k-\ell+1}}{k-\ell}\, dm(z)\\
  & \qquad \leq \left( \int_{D(0,r)} \left| \frac{T''(z)}{T'(z)} \right| \, dm(z) \right)^\frac{j_2}{k-\ell}
  \dotsb
  \left( \int_{D(0,r)} 
    \left| \frac{T^{(k-\ell+1)}(z)}{T'(z)} \right|^\frac{1}{k-\ell}\, dm(z) \right)^{\frac{(k-\ell) j_{k-\ell+1}}{k-\ell}}.
\end{align*}
The remaining part of the proof is similar to that above. This completes the proof of
Theorem~\ref{Local-Disc-thm-general}.

%%%%%%%%%%%%%%%%%%%%%%%
%%%% ---- SECTION ---- %%%%
%%%%%%%%%%%%%%%%%%%%%%%

\section{Sharpness discussion} \label{sec:examples}

The following examples illustrate the sharpness of Theorem~\ref{Local-Disc-thm-general}.

%%%%%%%%%%%%%%%%%%%%%%%%
%%%% ---- EXAMPLE ---- %%%%
%%%%%%%%%%%%%%%%%%%%%%%%

\begin{example}
For $\alpha> 1$, let
	$$
	a(z)=\frac{1-\alpha^2}{4z^2}-\alpha^2z^{2\alpha-2},\quad \Re(z)>0.
	$$
Then $a$ is analytic in the right half-plane, and $f''+af=0$ has linearly 
independent zero-free solutions
	$$
	f_j(z)=z^\frac{1-\alpha}{2}\exp\left((-1)^{j+1}z^\alpha\right),\quad j=1,2.
	$$
The function $T(z)=(1+z)/(1-z)$ maps $\D$ onto the right half-plane, and its is clear that
the Schwarzian derivative  vanishes identically. Moreover, by \eqref{eq:de2g}, the functions
\begin{equation*} 
  \begin{split}
    g_j(z)&=f_j(T(z))\, T'(z)^{-1/2}\\
    &= \frac{1}{\sqrt{2}} \, (1-z)^\frac{1+\alpha}{2}(1+z)^\frac{1-\alpha}{2}\exp\left((-1)^{j+1} \left(\frac{1+z}{1-z}\right)^\alpha\right),
    \quad j=1,2,
  \end{split}
\end{equation*}
are linearly independent zero-free solutions of $g''+bg=0$, where
\begin{equation*}
  \begin{split}	
    b(z)=a\big(T(z)\big)\, T'(z)^2+S_T(z)/2 %=A(T(z))T'(z)^2\\	
    =\frac{1-\alpha^2}{(1-z^2)^2}-\alpha^2\frac{(1+z)^{2\alpha-2}}{(1-z)^{2\alpha+2}},\quad z\in\D.
  \end{split}
\end{equation*}
From \eqref{conformal-change},
\begin{align*}
  \int_{T(D(0,r))} |a(z)|^\frac{1}{2}\, \frac{dm(z)}{| T'( T^{-1}(z) )|}
  & = \int_{D(0,r)} \big|a(T(z)) \, T'(z)^{2} \big|^\frac{1}{2}\, dm(z)
  = \int_{D(0,r)}|b(z)|^\frac{1}{2}\, dm(z)\\
  &\asymp  \int_{D(0,r)}\frac{dm(z)}{|1-z|^{\alpha+1}} \asymp \frac{1}{(1-r)^{\alpha-1}},
  \quad r\to 1^-.
\end{align*}
Meanwhile, the zeros of $g_1+g_2=(g_1/g_2+1)g_2$ are the points $z_n\in\D$ at which
$$
\exp\left(2\left(\frac{1+z_n}{1-z_n}\right)^\alpha\right)=-1=e^{\pi i},
$$
or equivalently 
\begin{equation*} 
  \left(\frac{1+z_n}{1-z_n}\right)^\alpha=\frac{(2n+1)\pi i}{2}=:w_n,\quad n\in\Z.
\end{equation*}
In particular, the points $w_n$ are located on the imaginary axis.
This means that the points $(1+z_n)/(1-z_n)$ are located on a~finite number of rays on the right half-plane emanating from the origin,
which in turn implies that the points~$z_n$ lie in a~Stolz angle with vertex at~$1$. Thus
	$$
	1-|z_n|\asymp |1-z_n|=\left|1-\frac{w_n^{1/\alpha}-1}{w_n^{1/\alpha}+1}\right|
	=\frac{2}{\big|w_n^{1/\alpha}+1\big|}\asymp \frac{1}{|n|^{1/\alpha}+1},
	\quad n\in\Z,
	$$
where the comparison constants are independent of $n$.
It follows that the small counting function $n(r)$ for the points $\{z_n\}$ satisfies 
	$
	n(r)\asymp (1-r)^{-\alpha},
	$
so that
	$$
	N\Big(T\big(D(0,s(r)\big),0,f_1+f_2\Big)=N\big(s(r),0,g_1+g_2\big) \asymp\int_0^{s(r)}\frac{n(t)}{t}\, dt\asymp \frac{1}{(1-r)^{\alpha-1}},
	\quad r\to 1^-.
	$$
This shows that Theorem~\ref{Local-Disc-thm-general} is sharp up to a~multiplicative constant in this case.
\hfill $\diamond$
\end{example}

%%%%%%%%%%%%%%%%%%%%%%%%
%%%% ---- EXAMPLE ---- %%%%
%%%%%%%%%%%%%%%%%%%%%%%%

\begin{example} \label{ex:jh}
Let $a_0,\ldots, a_{k-2}\in\R\setminus\{0\}$ be such that the characteristic equation
	$$
	r^k+a_{k-2}r^{k-2}+\cdots+a_1r+a_0=0
	$$
has $k$ distinct roots $r_1,\ldots,r_k\in\C\setminus\{0\}$. Then the functions $f_j(z)=e^{r_jz}$,
$j=1,\ldots,k$, form a zero-free solution base for \eqref{eq:fdek}
with constant coefficients. 
For $\alpha\in (1,2]$, let
	$$
T(z) = \bigg(\frac{1+z}{1-z}\bigg)^\alpha, \quad z\in\D.
	$$
Then $T$ maps $\D$ onto the sector $|\arg(z)|<\alpha\pi/2$ for $\alpha\in (1,2)$, and onto~$\C$ 
minus the real interval $(-\infty,0]$ for $\alpha=2$. 
Now the functions $g_j=(f_j \circ T )\, (T')^{(1-k)/2}$, $j=1,\ldots,k$, form a~zero-free solution base for \eqref{eq:gdek2}
in $\D$. From \eqref{conformal-change}
we find
\begin{equation}\label{coeffs}
  \begin{split}
    \int_{T(D(0,r))}|a_j|^\frac{1}{k-j}\, \frac{dm(z)}{| T'( T^{-1}(z) )|}
    \asymp \frac{1}{(1-r)^{\alpha-1}},
    \quad r\to 1^-,
  \end{split}	
\end{equation}
for $j\in \{0,\ldots,k-2\}$.

Let $f$ be a non-trivial linear combination of at least two exponential terms $f_j$. Without loss of generality,
we may suppose that $f=C_1f_1+\cdots + C_mf_m$, where $2\leq m\leq k$ and $C_1,\ldots,C_m\in\C\setminus\{0\}$.
Let 
\begin{equation*}
  g = C_1g_1+\cdots + C_mg_m
  = \big((C_1(f_1 \circ T)+\cdots + C_m (f_m \circ T)\big)(T')^{(1-k)/2}
\end{equation*}
denote the corresponding solution of~\eqref{eq:gdek2}.

Let $W=\{\overline{r}_1,\ldots,\overline{r}_m\}$, and let $\operatorname{co}(W)$ denote the convex hull of $W$. 
Then $\operatorname{co}(W)$ is either a line segment or a closed convex polygon in $\C$. Let 
$\Theta\subset(-\pi,\pi]$ denote the set of angles that the outer normals of $\operatorname{co}(W)$
form with the positive real axis. If $\operatorname{co}(W)$ has $s$ vertex points, then
it has $s$ outer normals, and $\Theta$ has $s$ elements, say 
\begin{equation*}
  \Theta=\big\{\theta_1,\ldots,\theta_s\big\},
  \quad
  -\pi<\theta_1<\theta_2<\cdots <\theta_s\leq \pi.
\end{equation*}
 For example, if ${r}_1,\ldots,{r}_m\in\R$, then $\Theta=\{\pm \pi/2\}$.
 In general $2\leq s\leq m$, and if $s=m$, 
then each point $\overline{r}_j$ is a vertex point of $\operatorname{co}(W)$. 
Set $\theta_{s+1}=\theta_1+2\pi$. Since clearly $\theta_{j+1}-\theta_j\leq \pi$ for all $j\in\{1,\ldots,s\}$,
and since $\sum_{j=1}^s (\theta_{j+1}-\theta_j)=2\pi$, it follows that at least one of the rays $\arg(z)=\theta_j$ 
lies entirely in $T(\D)$. We also point out that,
for a suitable set of roots ${r}_1,\ldots,{r}_m$, all of the rays $\arg(z)=\theta_j$ lie in $T(\D)$. 

Based on the work of P\'olya and Schwengeler in the 1920's, we state some facts about the zero
distribution of the exponential sum $f$. The exact references as well as proofs can be found in \cite{HW}.
For any $\varepsilon>0$, the zeros of~$f$ are in the union of $\varepsilon$-sectors 
$W_j=\{z\in\C : |\arg(z)-\theta_j|<\varepsilon\}$, with finitely many possible exceptions. In fact, the zeros of $f$
are in logarithmic strips around the rays $\arg(z)=\theta_j$. Each sector $W_j$ is zero-rich
in the sense that the number of zeros in $W_j\cap D(0,r)$ is asymptotically comparable to $r$.
In particular, the exponent of convergence for the zeros of $f$ in each sector $W_j$ is equal to one, same
as the order of~$f$. 

Let $\arg(z)=\theta_j$ be one of the rays that lies in $T(\D)$. Taking $\varepsilon>0$ small
enough, the sector $W_j$ lies in $T(\D)$ as well. The pre-image of $W_j$ is a circular wedge
in $\D$ having vertices of opening $\varepsilon/\alpha$ at the points $z=\pm 1$. Thus all zeros of 
$g$ are in such wedges, except
possibly finitely many. The zeros of~$g$ can accumulate to $1$ and nowhere else. Since $g$ has
Nevanlinna order $\alpha-1$ and finite type, it follows that
	$$
	N(r,0,g)\leq T(r,1/g)=T(r,g)+O(1)= O\big( (1-r)^{1-\alpha} \big),
	\quad r\to 1^-.
	$$   
Combining this with \eqref{coeffs} shows that in this case Theorem~\ref{Local-Disc-thm-general} is sharp up to a~multiplicative constant. In addition, since the functions $f_1,\ldots,f_k$ are zero-free, the second sum in
Theorem~\ref{Local-Disc-thm-general} involving the linear combinations $f_j+f_k$ is necessary. 
\hfill $\diamond$
\end{example}

%%%%%%%%%%%%%%%%%%%%%%%
%%%% ---- SECTION ---- %%%%
%%%%%%%%%%%%%%%%%%%%%%%

\section{Proof of Theorem~\ref{prop:new}} \label{sec:bases}

The proof relies on elementary properties of Wronskian determinants, which can be found, for example, in \cite[Sec.~1.4]{L:1993}.
We first show that $W(f_1^{k-1}, \, f_1^{k-2} f_2^{\phantom{k}}, \dotsc, f_1^{\phantom{k}} f_2^{k-2}, \, f_2^{k-1})$ is a~non-zero
complex constant, in which case
$\{f_1^{k-1}, \, f_1^{k-2} f_2^{\phantom{k}}, \dotsc, f_1^{\phantom{k}} f_2^{k-2}, \, f_2^{k-1}\}$ forms a~solution base
of \eqref{eq:fdek} with analytic coefficients by \cite[Propositions~1.4.6 and 1.4.8]{L:1993}. 
In fact, we prove that
\begin{equation} \label{eq:newclaim}
  W(f_1^{k-1}, \, f_1^{k-2} f_2^{\phantom{k}}, \dotsc, f_1^{\phantom{k}} f_2^{k-2}, \, f_2^{k-1}) = c_k \, W(f_1,f_2)^{s_k},
\end{equation}
where $W(f_1,f_2)\in\C\setminus\{0\}$ and
\begin{equation*}
  c_k = \prod_{j=2}^{k-1} j^{k-j} = 2^{k-2} 3^{k-3} \dotsb (k-1), \quad
  s_k = \sum_{j=1}^{k-1} j = \frac{k(k-1)}{2}.
\end{equation*}
We proceed by induction.
The identity~\eqref{eq:newclaim} is clearly true for $k=2$ as both sides reduce to $W(f_1,f_2)$.
Suppose that \eqref{eq:newclaim} is valid for some $k\geq 2$. It is well-known that
$w=f_1/f_2$ is a~locally univalent meromorphic function such that $w' = -W(f_1,f_2)/f_2^2$.
Then
\begin{align*}
  W(f_1^k, f_1^{k-1} f_2^{\phantom{k}}, \dotsc, f_1^{\phantom{k}} f_2^{k-1}, f_2^k)
  & = \big( f_2^k \big)^{k+1} \, W\Big( w^k, w^{k-1}, 
  \dotsc, w, 1 \Big)\\
  & = \big( f_2^k \big)^{k+1} \, (-1)^k \, W\Big( (w^k)' , (w^{k-1})', 
  \dotsc, w' \Big)\\
  & = \big( f_2^k \big)^{k+1} \, (-1)^k \, W\Big( k w^{k-1} w'  , (k-1) w^{k-2} w', 
  \dotsc, w' \Big),
\end{align*}
and the substitution back gives 
\begin{align*}
  & W(f_1^k, f_1^{k-1} f_2^{\phantom{k}}, \dotsc, f_1^{\phantom{k}} f_2^{k-1}, f_2^k)\\
  & \qquad = \big( f_2^{k+1} \big)^k \, W\bigg( k \, \frac{f_1^{k-1}}{f_2^{k-1}} \cdot \frac{W(f_1,f_2)}{f_2^2}  , 
  (k-1) \, \frac{f_1^{k-2}}{f_2^{k-2}} \cdot \frac{W(f_1,f_2)}{f_2^2},
  \dotsc, \frac{W(f_1,f_2)}{f_2^2} \bigg)\\
  & \qquad = W(f_1,f_2)^k \, W\big( k f_1^{k-1}   , 
  (k-1) f_1^{k-2} f_2^{\phantom{k}},
  \dotsc, f_2^{k-1} \big)\\
  & \qquad = k! \, W(f_1,f_2)^k \, W\big( f_1^{k-1},  f_1^{k-2} f_2, \dotsc, f_2^{k-1} \big).
\end{align*}
The induction hypothesis \eqref{eq:newclaim} gives
\begin{align*}
  W(f_1^k, f_1^{k-1} f_2^{\phantom{k}}, \dotsc, f_1^{\phantom{k}} f_2^{k-1}, f_2^k)
  &  = k! \, W(f_1,f_2)^k \, c_k \, W(f_1,f_2)^{s_k} = c_{k+1} \, W(f_1,f_2)^{s_{k+1}}.
\end{align*}
Therefore \eqref{eq:newclaim} holds for all $k\geq 2$.

Let $h_1,\dotsc, h_{k-1}$ be functions such that each is either $f_1$ or $f_2$. 
The products $h_1\dotsb h_{k-1}$ give a~complete description for
functions in the solution base obtained above, and hence
\begin{equation} \label{eq:aa}
  (h_1\dotsb h_{k-1})^{(k)} + a_{k-2} (h_1\dotsb h_{k-1})^{(k-2)} + \dotsb + a_1 (h_1\dotsb h_{k-1})' +a_0 h_1\dotsb h_{k-1} = 0,
\end{equation}
for any choices of $h_1,\dotsc, h_{k-1}$.
Recall that the coefficients $a_0,\dotsc, a_{k-2}$ are uniquely determined by the solution base.
We compare the representation
\begin{equation} \label{eq:glr1}
  (h_1\dotsb h_{k-1})^{(k)} = \sum \frac{k!}{s_1! \dotsb s_{k-1}!} \, h_1^{(s_1)} \dotsb h_{k-1}^{(s_{k-1})},
\end{equation}
obtained by the general Leibniz rule, to the other terms in \eqref{eq:aa}. The sum in~\eqref{eq:glr1} extends 
over all non-negative integers $s_1,\dotsc, s_{k-1}$ for which $s_1+\dotsb + s_{k-1}=k$.
Similarly,
\begin{equation} \label{eq:glr2}
  (h_1\dotsb h_{k-1})^{(k-2)} = \sum \frac{(k-2)!}{j_1! \dotsb j_{k-1}!} \, h_1^{(j_1)} \dotsb h_{j-1}^{(j_{k-1})},
\end{equation}
where the sum is taken over all non-negative integers $j_1,\dotsc, j_{k-1}$
for which $j_1+\dotsb + j_{k-1}=k-2$. The sum~\eqref{eq:glr2} contains terms which are
exceptional in relation to the other terms. For example, consider the term corresponding to indices
$j_1=\dotsb = j_{k-2}=1$ and $j_{k-1}=0$. Since $j_1+\dotsb+j_{k-1}=k-2$, the analogous
representations for $(h_1\dotsb h_{k-1})^{(n)}$, $0\leq n \leq  k-3$, do not have terms of
the type $h_1' h_2' \dotsb h_{k-2}' h_{k-1}$. This means that all other terms of this type
are obtained from \eqref{eq:glr1} by using the fact
\begin{equation*}
  h_i^{(n)} = (h_i'')^{(n-2)} = - (ah_i)^{(n-2)} = - \Big( a^{(n-2)} h_i + \dotsb + a h_i^{(n-2)} \Big), \quad i=1,\dotsc, k-1, \quad n\geq 2.
\end{equation*}
There are $k-1$ possible sets of indices in \eqref{eq:glr1} which are transformed to $(1,\dotsc, 1,0)$ in this way, and they are
\begin{equation*}
  (3,1,1\dotsc,1, 1,0), \quad (1,3,1,\dotsc,1, 1,0), \quad \dotsc, \quad (1,1,1,\dotsc, 1,3,0), \quad (1,1,1,\dotsc, 1,1,2).
\end{equation*}
By a~careful comparison of \eqref{eq:glr1} and \eqref{eq:glr2}, and then taking \eqref{eq:aa} into account, we
see that the coefficient of $h_1' h_2' \dotsb h_{k-2}' h_{k-1}$ must satisfy
\begin{equation*}
  - a \left( (k-2) \, \frac{k!}{3! \,1! \, 1! \dotsc 0!} + \frac{k!}{1! \dotsb 1! \, 2!} \right)
  + a_{k-2} \, \frac{(k-2)!}{1! \dotsb 1! \, 0!} = 0.
\end{equation*}
Solving this identity for $a_{k-2}$ gives \eqref{eq:akm2} and completes the proof.

We point out that the Theorem~\ref{prop:new} admits the following 
meromorphic counterpart: Suppose that $f_1$ and $f_2$ are linearly independent meromorphic
solutions of $f''+af=0$, where the coefficient $a$ is meromorphic. For any $k\geq 2$, the functions
$f_1^{k-1}, \, f_1^{k-2} f_2^{\phantom{k}}, \dotsc, f_1^{\phantom{k}} f_2^{k-2}, \, f_2^{k-1}$
are linearly independent meromorphic solutions of \eqref{eq:fdek} with meromorphic coefficients
$a_0,\dotsc, a_{k-2}$ whose poles are among the poles of $f_1$ and $f_2$,
ignoring multiplicities. The identity~\eqref{eq:akm2} extends also to the meromorphic case.

%%%%%%%%%%%%%%%%%%%%%%%%%%%
%%%% ---- BIBLIOGRAPHY ---- %%%%
%%%%%%%%%%%%%%%%%%%%%%%%%%%

\end{document}